\documentclass{daj}

\usepackage{upgreek, amssymb, amsthm, amsmath}
\usepackage{url}
\usepackage{hyperref}\hypersetup{colorlinks=true}
\usepackage[margin=3.5 cm]{geometry}
\usepackage{hyperref} 

\newtheorem{theorem}{Theorem}[section]
\newtheorem{fact}[theorem]{Fact}
\newtheorem{prop}[theorem]{Proposition}

\newtheorem{coro}[theorem]{Corollary}
\newtheorem{lem}[theorem]{Lemma}
\newtheorem{conj}[theorem]{Conjecture}
\newtheorem{obs}[theorem]{Observation}
\newtheorem{defn}[theorem]{Definition}

\numberwithin{equation}{section}

\newcommand\be{\begin{equation}}
\newcommand\ee{\end{equation}}
\newcommand\bea{\begin{eqnarray}}
\newcommand\eea{\end{eqnarray}}

\newcommand\R{\mathbb{R}}

\newcommand\N{\mathbb{N}}
\newcommand\E{\mathbb{E}}

\newcommand\nn{\nonumber}

\newcommand{\OR}{\text{OR}}
\newcommand{\AND}{\text{AND}}
\newcommand{\MAJ}{\text{MAJ}}

\newcommand{\poly}{\text{poly}}

\newcommand{\Inf}{\text{Inf}}

\newcommand{\DT}{\text{DT}}

\newcommand{\I}{\textbf{I}}
\newcommand{\bls}{\text{bs}}

\newcommand{\adeg}{\widetilde{\deg}}

\newcommand{\sens}{\text{sens}}
\newcommand{\cert}{\text{cert}}
\newcommand{\MAF}{\text{MAF}}

\newcommand{\Rel}{\text{\fontfamily{lmss}\selectfont
Rel}}

\dajAUTHORdetails{%
  title = {Relationships between the number of inputs and other complexity measures of Boolean functions}, 
  author = {Jake Wellens},
  plaintextauthor = {Jake Wellens},
  runningtitle = {The number of inputs and other complexity measures}, 
  keywords = {complexity measures, Boolean functions},
}   

\dajEDITORdetails{%
   year={2022},
   number={19},
   received={29 May 2020},   
   revised={17 April 2021},    
   published={23 December 2022},  
   doi={10.19086/da.57741},       
}   

\begin{document}

\begin{frontmatter}[classification=text]

\title{Relationships between the number of inputs and other complexity measures of Boolean functions}

\author[jake]{Jake Wellens}

\begin{abstract}
We generalize and extend the ideas in a recent paper of Chiarelli, Hatami and Saks to prove new bounds on the number of relevant variables for Boolean functions in terms of a variety of complexity measures. Our approach unifies and refines all previously known bounds of this type. We also improve Nisan and Szegedy's well-known inequality $\bls(f) \leq \deg(f)^2$ by a constant factor, thereby improving Huang's recent bound on the sensitivity conjecture by the same constant. 
\end{abstract}
\end{frontmatter}

\section{Introduction}

Is a Boolean function $f: \{0,1\}^n \to \{0,1\}$ necessarily ``complex'' simply because it takes many input variables?\footnote{Of course, one has to count only the \textit{relevant} inputs, ignoring any dummy variables which $f$ does not actually need.} In 1983, Simon \cite{Simon} answered this question in the affirmative, showing that the number of relevant variables $n(f)$ of a Boolean function $f$ is bounded above by $s(f)4^{s(f)}$, where $s(f)$ is the sensitivity of $f$. Combined with earlier work by Cook, Dwork and Rieschuk \cite{CD} showing a $\Omega(\log s(f))$ lower bound on the parallel (CREW-PRAM) complexity of $f$, Simon's theorem implies that any function with $n$ relevant inputs takes $\Omega(\log \log n)$ time to evaluate on the worst case input, even with an \textit{unbounded} number of processors working in parallel. A decade later, Nisan and Szegedy \cite{NS} proved a similar upper bound on $n(f)$ in terms of the degree of $f$, namely 
\be \label{Nisan_Szeg}
n(f) \leq \deg(f) \cdot 2^{\deg(f) - 1},
\ee
which was very recently improved to 
\be \label{CHS_bd}
n(f) \leq 6.614 \cdot 2^{\deg(f)}
\ee
by Chiarelli, Hatami and Saks \cite{CHS}, which is tight up to the constant factor. While the proof of (\ref{Nisan_Szeg}) uses the average sensitivity or total influence $\I[f]$ as a potential function, the proof of (\ref{CHS_bd}) requires a potential based on a local version of degree for each coordinate $i$, namely $\deg_i(f) = \deg(f(x) - f(x \oplus e_i))$. 

In this paper, we generalize the measure $\deg_i$ to a class of measures with the same essential properties. This provides a common framework for proving nearly-tight bounds on $n(f)$ in terms of various complexity measures. In particular, we give short, unified proofs of the theorems of Simon, Nisan-Szegedy and Chiarelli-Hatami-Saks, as well as a variety of new and improved bounds, which we summarize in the following theorem.

\begin{theorem}\label{main_intro}
For any Boolean function $f$,
\bea 
n(f) &\leq& 4.394 \cdot 2^{\deg(f)} \label{CHS_new} \\ 
n(f) &\leq& \frac{1}{2}\cdot 4^{C(f)}.  
\eea
Moreover, if $f$ is monotone, then
\be \nn
n(f) \leq \min\left\{1.325 \cdot 2^{\deg(f)}, \, \, \frac{1}{2}\cdot 4^{s(f)}, \, \, \frac{1}{4}\cdot 2^{\emph{\DT}(f)} + 2   \right\}.
\ee
\end{theorem}

Prior to our work, the best bounds on $n(f)$ in terms of $s(f)$ or $C(f)$ came from Simon's theorem \cite{Simon}, namely $n(f) \leq s(f)\cdot 4^{s(f)}$ (and therefore $n(f) \leq C(f) \cdot 4^{C(f)}$). While the bound in (\ref{CHS_new}) is tight up to the constant (in \cite{CHS}, the authors also give an example with $n(f) \geq 1.5 \cdot 2^{\deg(f)} - 2$), the gaps between the other bounds in Theorem \ref{main_intro} and the best-known examples are still polynomial in the relevant complexity measure. In fact, the largest-known separations all come from Wegener's monotone address function \cite{Weg}, which is a monotone function with 
$$n(f) = \Theta\left(\frac{2^{\DT(f)}}{\sqrt{\DT(f)}}  \right) = \Theta\left(\frac{2^{\deg(f)}}{\sqrt{\deg(f)}}\right) = \Theta\left(\frac{4^{s(f)}}{\sqrt{s(f)}}\right).$$
We believe that closing the gap for any of these complexity measures is a fundamental challenge left open by our work. 

The proof of inequality (\ref{CHS_new}) relies in part on improved upper bounds on block sensitivity ($\bls(\cdot)$) by degree for Boolean functions with degree up to $d = 14$, which we obtain via a natural LP relaxation. The bounds from the LP appear to suggest that, roughly, $\bls(f) \leq 0.6 \cdot \deg(f)^2$, which would be a constant factor improvement over the previously best-known bound of $\bls(f) \leq \deg(f)^2$, due to Nisan and Szegedy \cite{NS} (and later optimized by Tal \cite{Tal}). In Section \ref{sec_bs_d} of this paper, we obtain such an improvement for all $f$ (with a constant $\sim 0.82$ instead of $0.6$):
\begin{theorem}\label{intro_improved_bs_deg}
For any Boolean function $f$, 
\be 
\emph{\bls}(f) \leq \sqrt{2/3}\cdot \deg(f)^2 + 1.
\ee
\end{theorem}
Theorem \ref{intro_improved_bs_deg} has a number of consequences. First, it explains why standard tensorization techniques have been unable to produce a quadratic separation between degree and block sensitivity (the best-known example achieves $\bls(f) = \deg(f)^{\log_3(6)}$). It also improves the best-known relationships between many other pairs of complexity measures by the same constant factor, most notably Huang's recent bound on block sensitivity in terms of sensitivity \cite{Huang}. 

The proof of Theorem \ref{intro_improved_bs_deg} resembles the original two-step approach of Nisan and Szegedy: (1) symmetrize the Boolean function $f : \{0,1\}^n \to \{0, 1\}$ to obtain a univariate polynomial $p$ of the same degree, bounded on $[0,1]$, and (2) relate the derivative of $p$ to $\bls(f)$ and appeal to an inequality from approximation theory to bound $\|p'\|_{L^{\infty}[0,1]}$ in terms of $\|p\|_{L^{\infty}[0,1]}$ and $\deg(p)$. However, our approach differs in two important ways, each of which is critical for obtaining the tighter constant. First, we target a different polynomial $p$, obtained via the noise operator rather than by averaging over the symmetric group. The norm of this polynomial is tightly controlled on the whole interval, rather than merely on a discrete net of points like Nisan and Szegedy's symmetrization, which avoids incurring a non-sharp constant factor in this step. Second, we are able to relate the \emph{second derivative} of this polynomial to block sensitivity, which enables us to appeal to the Markov Brothers' inequality and pick up a factor which is less than 1. This step relies on $f$ taking Boolean values on \emph{three} slices of the Boolean hypercube, whereas the original proof only requires two. \\

\noindent \textbf{Organization:} We provide definitions of all the relevant complexity measures in Section \ref{sec_prelim}. Then in Section \ref{sec_variables}, we first give a high-level overview of our method for proving bounds on $n(f)$, and then define our generalized coordinate measures in \ref{sec_RRCM}. The rest of Section \ref{sec_variables} is devoted to proving those bounds for a variety of complexity measures. In Section \ref{sec_bs_d}, we prove Theorem \ref{intro_improved_bs_deg}. Finally in Section \ref{sec_future_directions}, we discuss some open problems related to our work.

\section{Preliminaries}\label{sec_prelim}

All functions $f$ in this paper will be assumed to be Boolean valued on $\{0,1\}^n$. We will refer to the input variables of such functions either by $x_i$ or simply by the index $i$, for each $i \in \{1, \dots, n\}  =: [n]$. We define $\Rel(f)$ to be the set of relevant variables/coordinates for $f$, namely those $i \in [n]$ for which there exists a pair of inputs $(x, x')$ such that $x_j = x'_j$ for all $j \neq i$ and $f(x) \neq f(x')$. We write $\delta_i(f)$ for the indicator function of whether $i \in \Rel(f)$. We also define $$n(f) := |\Rel(f)| = \sum_{i \in [n]}\delta_i(f)$$ to be the number of relevant variables for $f$. We briefly review the definitions of the various complexity measures mentioned in the introduction (see e.g. the survey \cite{BdW} for a more comprehensive treatment of these and other related measures.)

\subsection{Complexity measures}

For each $f$ there is a corresponding multilinear polynomial over $\{0,1\}^n$, which we call the multilinear polynomial expansion of $f$. The degree of this polynomial is the \textbf{degree} of $f$, denoted $\deg(f)$.\\

For any string $x \in \{0,1\}^n$ and a subset $S \subseteq [n]$, we let $x^S$ denote the string obtained by flipping the bits of $x$ belonging to $S$ and leaving the rest alone. If $S = \{i\}$, we simply write $x^i$ to denote $x$ with the $i$th bit flipped. If $f(x) \neq f(x^i)$, we say that $f$ is sensitive to $i$ at $x$. The sensitivity of $f$ at an input $x$, denoted $s_x(f)$, is the number of $i \in [n]$ for which $f$ is sensitive to $i$ at $x$. The maximum of $s_x(f)$ over all $x \in \{0,1\}^n$ is called the \textbf{sensitivity} of $f$ and is denoted $s(f)$. The 1-sensitivity (resp. 0-sensitivity) of $f$, denoted $s^1(f)$ (resp. $s^0(f)$), is the maximum of $s_x(f)$ over all inputs $x$ with $f(x) = 1$ (resp. 0). \\

The block sensitivity at the point $x$ of a Boolean function $f: \{0,1\}^n \to \{0,1\}$, denoted $\bls_x(f)$, is the maximum number $k$ such that there exist $k$ disjoint sets $B_1, \dots, B_{k} \subseteq [n]$ (called blocks) with the property that
$$f(x) \ne f(x^{B_i}), \, \, \text{for } i = 1, \dots, k.$$
We then define the \textbf{block sensitivity} of $f$ to be the maximum value of $\bls_x(f)$ over all $x \in \{0,1\}^n$, and we denote it by $\bls(f)$. \\

The certificate complexity at the point $x$ of a Boolean function $f$, denoted $C_x(f)$, is the size of the smallest set $S \subseteq [n]$ with the property that $f$ is constant on the subcube of points which agree with $x$ on $S$, i.e. $\{y: y_i = x_i \text{ for all } i \in S\}$. The \textbf{certificate complexity} of $f$, denoted $C(f)$, is then defined as the maximum value of $C_x(f)$ over all $x \in \{0,1\}^n$. Also, let $C_{\min}(f) := \min_{x \in \{0,1\}^n} C_x(f)$. By analogy with $s^0(f)$ and $s^1(f)$, we can also define $C^0(f)$, $C^1(f)$, $C_{\min}^0(f)$ and $C_{\min}^1(f)$ in the obvious way.\\

It is easy to show (see \cite{Nisan}) that $s(f) \leq \bls(f) \leq C(f)$ always, and that for any \textit{monotone} Boolean function, the three measures actually coincide:
\be \label{s=C} s(f) = \bls(f) = C(f). \ee

The \textbf{decision tree depth} of $f$, denoted $\DT(f)$, is defined to be the minimum cost of any deterministic, adaptive query algorithm which always computes $f$ correctly. (The cost of such an algorithm is defined to be the maximal number of queries used by the algorithm to compute $f(x)$, taken over all $x \in \{0,1\}^n$.) \\

The \textbf{$\varepsilon$-approximate degree} of a Boolean function $f: \{0,1\}^n \to \{0,1\}$ is the smallest $d$ for which there exists a degree $d$ (multilinear) polynomial $p(x_1, \dots, x_n)$ such that 
$$|p(x) - f(x)| \leq \varepsilon\, \, \, \text{ for all } x \in\{0,1\}^n,$$
and we denote this quantity by $\widetilde{\deg}_\varepsilon(f).$ If we omit the $\varepsilon$ and simply write $\adeg(f)$, it should be understood to mean $\widetilde{\deg}_{1/3}(f).$ This is the canonical and somewhat arbitrary choice -- replacing $1/3$ by any other constant can only change the value of $\adeg(f)$ by a constant factor.

\subsection{Fourier influence}

Any function $f : \{0,1\}^n \to \{0,1\}$ can also be viewed as a function from $\{\pm 1\}^n \to \{\pm 1\}$ via the obvious affine transformation, and then expressed as a linear combination $f(x) = \sum_{S \subseteq [n]} \hat{f}(S) x^S$ of monomials $x^S = \prod_{i \in S} x_i$. The coefficients $\hat{f}(S)$ are called the Fourier coefficients of $f$. For each coordinate $i$, we define the \emph{$i$th coordinate influence} of a function $f$, denoted $\Inf_i[f]$, as
\be\nn \Inf_i[f] := \Pr_{x \sim \{\pm 1\}^n}[f(x) \neq f(x^i)] \ee and the \emph{total influence} of $f$, denoted by $\I[f]$, is defined to be $\sum_{i = 1}^n \Inf_i[f]$. We'll need the following well-known Fourier formulas for influence:
$$\Inf_i[f] = \sum_{S \ni i} \hat{f}(S)^2, \, \, \, \, \I[f] = \sum_{S \subseteq [n]} |S|\hat{f}(S)^2$$
as well as the following well-known fact about influence and restrictions:\footnote{We use the notation $f_{\alpha}$ for $\alpha \in \{0,1\}^H$ to denote the function obtained from $f$ by restricting the coordinates in $H$ according to the partial assignment $\alpha$.}
\begin{fact}\label{inf_avg} For any $i \in [n]$, and any set $H \subset [n]$ with $i \not \in H$,
$$\emph{\Inf}_i[f] = \E_{\alpha \sim \{0,1\}^H}[\emph{\Inf}_i[f_{\alpha}]].$$
\end{fact}

\section{Improved bounds on the number of variables}\label{sec_variables}

\subsection{Overview}\label{sec_overview}

Our goal is to develop a unified framework for proving bounds on $n(f)$ in terms of various complexity measures like $\deg(f)$, $s(f)$ and $C(f)$. A critical ingredient in each proof is a certain coordinate-centered analogue $m_i$ of each complexity measure $m$ (see Definition \ref{RRCM}), which is constructed in such a way that the corresponding potential function 
\be\label{potential}
\textbf{M}(f) := \sum_{i \in [n]} \frac{\delta_i(f)}{2^{m_i(f)}}.
\ee
is guaranteed to obey the inequality
\be\label{H_rest}
\textbf{M}(f) \leq \sum_{i \in H}\frac{\delta_i(f)}{2^{m_i(f)}} + \E_{\alpha \sim \{0,1\}^H}[\textbf{M}(f_\alpha)]
\ee
for any $H \subseteq [n]$. Inequality (\ref{H_rest}) enables us to bound $\textbf{M}(f)$ recursively by averaging over random restrictions, provided that we select the set of coordinates $H$ so that the restrictions $f_{\alpha}$ are guaranteed to have \textit{lower complexity}, in some sense. Upper bounds on $\textbf{M}(f)$ naturally yield exponential upper bounds on $n(f)$ in terms of $m(f)$. We make these definitions precise below in the next subsection, and each subsequent subsection describes a different implementation of the general strategy above, yielding new bounds.

\subsection{Restriction-reducing coordinate measures}\label{sec_RRCM}


\begin{defn}\label{RRCM}
We say a functional $m_i$ on the set of Boolean functions is a \emph{restriction reducing $i$-coordinate measure} if, for any $j \in [n]\setminus\{i\}$, and each $b \in \{0,1\}:$
\begin{itemize}
    \item[(1)] $m_i(f_{j=b}) \leq m_i(f)$
    \item[(2)] if $\delta_i(f) = 1$ and $\delta_i(f_{j=b}) = 0$, then $m_i(f_{j=1-b}) \leq m_i(f) - 1.$
\end{itemize}
\end{defn}

We denote by $\mathcal{R}_i$ the set of restriction reducing $i$-coordinate measures. We abuse notation sightly and write $\{m_i\} \in \mathcal{R}_i$ to denote that $m_i \in \mathcal{R}_i$ for each $i \in [n]$. In words, these properties mean that \emph{whenever fixing some value for a variable $x_j$ removes dependence on $x_i$, the restrictions must be smaller in $m_i$ measure than the original function.} This implies a kind of ``subharmonicity'' for $\delta_i(f)2^{-m_i(f)}$: 

\begin{fact}\label{restriction_single_fact}
Let $m_i \in \mathcal{R}_i$, and let $j \in [n] \setminus \{i\}$. Then
\be\label{restriction_single} \delta_i(f)2^{-m_i(f)} \leq \frac{\delta_i(f_{j=0})2^{-m_i(f_{j=0})} + \delta_i(f_{j=1})2^{-{m_i(f_{j=1})}}}{2}.\ee
\end{fact}

\begin{proof}
If $\delta_i(f_{j=0}) = \delta_i(f_{j=1}) = 1$, then property (1) of Definition \ref{RRCM} implies that both $2^{-m_i(f_{j=0})} \geq 2^{-m_i(f)}$ and $2^{-m_i(f_{j=1})} \geq 2^{-m_i(f)}$, which implies (\ref{restriction_single}). Otherwise, suppose without loss of generality that $\delta_i(f_{j=0}) = 0$ and $\delta_i(f_{j=1}) = 1$. Then property (2) of Definition \ref{RRCM} implies that $2^{-m_i(f_{j=1})} \geq 2 \cdot 2^{-m_i(f)}$ which also implies (\ref{restriction_single}). 
\end{proof}

Fact \ref{restriction_single_fact} extends easily by induction to larger restrictions:

\begin{fact}\label{m_rest}
For any $i \in [n]$ and any $H \subset [n]$ with $i \not \in H$, and any $\{m_i\} \in \mathcal{R}_i$,
\be \delta_i(f)2^{-m_i(f)} \leq \E_{\alpha \sim \{0,1\}^H}\left[\delta_i(f_\alpha)2^{-m_i(f_\alpha)}\right].\ee
\end{fact}
\begin{proof}
We proceed by induction on $|H|$. The base case $H = \{j\}$ is Fact \ref{restriction_single_fact}. For the inductive step, observe that if $\delta_i(f)2^{-m_i(f)} \leq \E_{\alpha \sim \{0,1\}^H}\left[\delta_i(f_\alpha)2^{-m_i(f_\alpha)}\right]$ holds for all $f$ with $H = H_1$ or $H_2$, then it holds for $H = H_1 \sqcup H_2$, since
\bea\nn
\delta_i(f)2^{-m_i(f)} &\leq& \E_{\alpha_1 \sim \{0,1\}^{H_1}}\left[\delta_i(f_{\alpha_1})2^{-m_i(f_{\alpha_1})}\right] \\\nn
&\leq& \E_{\alpha_1 \sim \{0,1\}^{H_1}}\left[\E_{\alpha_2 \sim \{0,1\}^{H_2}}\left[\delta_i(f_{\alpha_1, \alpha_2})2^{-m_i(f_{\alpha_1, \alpha_2})}\right]\right] \\
\nn &=& \E_{\alpha \sim \{0,1\}^{H_1 \sqcup H_2}}\left[\delta_i(f_\alpha)2^{-m_i(f_\alpha)}\right].
\eea 
\end{proof}

For any $\{m_i\} \in \mathcal{R}_i$, we can define the associated potential function $\textbf{M}$ via equation (\ref{potential}). By Fact \ref{m_rest}, $\textbf{M}$ satisfies the inequality (\ref{H_rest}) for any set $H \subseteq [n]$ of restricted coordinates. Next we introduce three explicit families\footnote{The reader may wonder why we do not define $\bls_i(f)$, the block sensitivity version of $\sens_i(f)$. This is indeed an RRCM, however, our arguments are not able to exploit the distinction between counting coordinates versus counting \emph{blocks} of coordinates, so we were not able to prove anything stronger with $\bls_i$ than what we can already show with $\sens_i$.} of RRCMs, the first of which ($\deg_i$) was introduced in \cite{CHS}:

\begin{defn}\label{RRCMs}
For each $i \in [n]$, define the $i$-coordinate measures
\bea
\label{deg_i}\deg_i(f) &:=& \deg(f(x) - f(x^i))\\
\label{sens_i}\emph{\sens}_i(f) &:=& \max_{\{x \, : f(x) \neq f(x^i)\}} s_x(f) + s_{x^i}(f) \\
\label{cert_i}\emph{\cert}_i(f) &:=& \max_{\{x \, : f(x) \neq f(x^i)\}} C_x(f) + C_{x^i}(f)
\eea

\end{defn}

The idea behind the measures in Definition \ref{RRCMs} is to keep track of a relevant quantity (e.g. sensitive edges) on \emph{both sides of a given edge} in the Boolean hypercube, thereby allowing one to exploit a simple property of the two-dimensional facets of the cube.\footnote{Color the vertices of the hypercube according their value under $f$; then inside any square face which is not monochromatic, any monochromatic edge has a perpendicular bichromatic edge. This is at the heart of Lemma \ref{rest_reduce_lem}, although we do not use this language in the proof.} 

\begin{lem}\label{rest_reduce_lem}
For each $i \in [n]$, the coordinate measures $\deg_i$, $\emph{\sens}_i$, and $\emph{\cert}_i$ all belong to $\mathcal{R}_i$. 
\end{lem}
\begin{proof}
Since $\deg(\cdot)$, $s_x(\cdot)$ and $C_x(\cdot)$ cannot possibly increase by restricting input variables, property (1) of Definition \ref{RRCM} is trivially satisfied for each of the coordinate measures in question. To see that (2) holds, we abbreviate $f_{j=b}$ by $f_b$ and assume without loss of generality that $\delta_i(f_0) = 0$. 

First we argue that $\deg_i(f_1) = \deg_i(f) - 1$. We can write $f(x) = x_jf_1(x) + (1-x_j)f_0(x)$. Since $x_i$ does not appear in $(1-x_j)f_0(x)$, it follows that $f(x) - f(x^i) = x_j(f_1(x) - f_1(x^i))$ from which it is clear that $\deg_i(f)  = 1 + \deg_i(f_1)$. 

Next we argue $\sens_i(f_1) = \sens_i(f) - 1$. Let $x$ be any input for which $f(x) \neq f(x^i)$, and let us write $y$ for the string which is $x$ with the $j$th bit omitted. Since $f_0$ does not depend on $i$, it must be that $f_0(y^i) = f_0(y)$. Therefore all such $x$ must have $x_j = 1$, so $f_1(y) = f(x) \neq f(x^i) = f_1(y^i)$. But then $j$ must be sensitive for $f$ at exactly one of $x^i$ or $x$, hence $s_x(f) + s_{x^i}(f) = s_x(f_1) + s_{x^i}(f_1) + 1$.

Finally we argue $\cert_i(f_1) = \cert_i(f) - 1$, which essentially follows from the previous paragraph. Indeed, as above, all $x$ for which $i$ is sensitive for $f$ must have  $x_j = 1$, and $j$ must be sensitive for exactly one of $x$ or $x^i$ -- suppose it is $x$ (wlog). Then any certificate for $f$ which agrees with $x$ must assign 1 to $x_j$, since if it were allowed to be flipped, the certificate could not make $f$ constant. The claim follows. 
\end{proof}

\begin{lem}\label{inf_bound}
Let $m_i$ be a restriction reducing $i$-coordinate measure and set $r := \min\{m_i(\chi_i), m_i(\neg \chi_i)\}$, where $\chi_i(x) = x_i$. Then for any Boolean function $f$,
\be \label{ineq_i}\delta_i(f)2^{-m_i(f)} \leq 2^{-r} \cdot \emph{\Inf}_i[f]. \ee
Hence $\emph{\textbf{M}}(f) \leq 2^{-r} \cdot \emph{\I}[f]$ and for any $k \in \N$, at most $\emph{\I}[f] \cdot 2^{k-r}$ relevant variables can have $m_i(f) \leq k$.
\end{lem} 
\begin{proof} We proceed by induction on $n(f)$. If $n(f) = 1$, then the corollary follows from the definition of $r$ and the fact that $\Inf_i[f] = \delta_i(f)$. Now suppose the desired inequality holds for all $f'$ with $n(f') < n(f)$, and we wish to show it holds for $f$ as well. Then by the induction hypothesis and Fact \ref{restriction_single_fact},
\be \delta_i(f)2^{-m_i(f)} \leq \frac{2^{-r} \cdot \Inf_i[f_{j=0}] + 2^{-r} \cdot \Inf_i[f_{j=1}]}{2} = 2^{-r} \cdot \Inf_i[f]\ee
where the final equality is Fact \ref{inf_avg}. If we sum this inequality over $i \in \Rel(f)$, we obtain
\be
\textbf{M}(f) = \sum_{k = 0}^{\infty}\frac{|\{j \in \Rel(f) \, : \, m_j(f) = k\}|}{2^k} \leq 2^{-r}\I[f]
\ee
which in particular implies that at most $\I[f] \cdot 2^{k-r}$ variables in $\Rel(f)$ have $m_i(f) \leq k$.
\end{proof}

\begin{obs}\emph{
Applying Lemma \ref{inf_bound} to the measures $\deg_i$ and $\sens_i$ immediately yields both Nisan-Szegedy's and Simon's theorems\footnote{Note that $\I[f] \leq \deg(f)$ and $\I[f] \leq s(f)$.}. Indeed, $\min\{\deg_i(x \mapsto x_i), \deg_i(x \mapsto \neg x_i)\} = 1$ and $\min\{\sens_i(x \mapsto x_i), \sens_i(x \mapsto \neg x_i)\} = 2$, so \bea n(f) &\leq& \I[f]\cdot 2^{\deg(f)-1}  \\ \label{NS+Simon} n(f) &\leq& \I[f] \cdot 4^{s(f)-1}.\eea
}
\end{obs}

\subsection{Degree}

Let $\textbf{D}(f) := \sum_{i \in [n]}\frac{\delta_i(f)}{2^{\deg_i(f)}}$, and for any $H \subseteq [n]$, let $\textbf{D}(H, f) = \sum_{i \in H}\frac{\delta_i(f)}{2^{\deg_i(f)}}$. For any $d \in \N$, let $\textbf{D}_d = \max_{\{f \, : \, \deg(f) \leq d\}}\textbf{D}(f)$. In \cite{CHS}, the authors argue that one can always find a set $H$ of $\leq \deg(f)^3$ coordinates such that (i) $\deg_i(f) = \deg(f)$ $\forall i \in H$ and (ii) $\deg(f_{\alpha}) < \deg(f)$ for all $\alpha \in \{0,1\}^H$. This implies $\textbf{D}_d \leq \frac{d^3}{2^d} + \textbf{D}_{d-1}$, and hence that $\textbf{D}(f) < \sum_{d = 1}^{\infty}\frac{d^3}{2^d} = 26$ for all $f$. Combined with the observation that $\textbf{D}_d \leq \frac{d}{2}$ (see Lemma \ref{inf_bound}), this yields Chiarelli, Hatami and Saks' final bound $\textbf{D}(f) \leq \frac{11}{2} + \sum_{d = 12}^\infty \frac{d^3}{2^d} \approx 6.614$.

In this subsection, we implement their argument in a slightly different way to obtain a slightly stronger bound. In particular, rather than choosing $H$ to be a minimal set of coordinates which covers all max degree monomials in $f$, we choose $H$ to be the variables in a \textit{single monomial} of $f$. Restricting this set of coordinates may not reduce the degree of $f$, but as shown below, it \textit{will} reduce the block sensitivity of $f$. Hence, as we'll want to induct on both degree and block sensitivity simultaneously, we define
$$\textbf{D}_{b, d} : = \max_{\substack{f \text{ with }\bls(f) \leq b  \\ \text{ and } \deg(f) = d }} \textbf{D}(f).$$
We also define $$B_d := \max_{\{f : \deg(f) = d\}}\bls(f).$$ The following lemma is originally due to Nisan and Smolensky (in unpublished work, see \cite{BdW}).

\begin{lem}\label{bs decr}
If $M$ is a monomial of degree $d = \deg(f)$ which appears in $f$ with non-zero coefficient, then for any assignment $\alpha: M \to \{0,1\}$, the restricted function $f_\alpha$ has \emph{$\bls(f_\alpha) \leq \bls(f) - 1.$ }
\end{lem}

\begin{proof}
Let us write any string $x \in \{0,1\}^n$ as $x = (x_M, y)$, where $x_M \in \{0,1\}^M$ and $y \in \{0,1\}^{[n]\setminus M}$. We claim that for any $(x_M, y)$, there is always a sensitive block for $f$ contained entirely in $M$. Indeed, for any $y$, the function $f(\cdot, y)$ has degree $d$, since nothing can cancel with the maximal monomial $\prod_{i \in M}x_i$. In particular, it is not constant, so for any input $x_M$, there is always at least one sensitive block for $f(\cdot, y)$ at $x_M$. Therefore, $\bls_{y}(f_{\alpha}) + 1 \leq \bls_{(\alpha, y)}(f)$, and the lemma follows.
\end{proof}

\begin{lem}\label{bd_recursive}
For each $b,d$ with $b \leq B_d$, we have
$$\emph{\textbf{D}}_{b, d} \leq d  \cdot2^{-d} + \max_{k \in \{1, \dots, d \}} \emph{\textbf{D}}_{b - 1, k} $$
\end{lem}

\begin{proof}
Suppose $f$ has $\deg(f) =  d$ and $\bls(f) \leq b$. Let $M$ be any degree $d$ monomial in $f$. Using (\ref{H_rest}), 
\be \textbf{D}(f) \leq \underbrace{|M|\cdot 2^{-d}}_{= \, d \cdot 2^{-d}} + \underset{\alpha \sim \{0,1\}^M}{\E}[\textbf{D}(f_\alpha)].
\ee
By Lemma \ref{bs decr}, each $f_{\alpha}$ has $\bls(f_\alpha) \leq b - 1$. Since $\textbf{D}_{b, d}$ is monotone in $b$, it follows that for each $\alpha$, $\textbf{D}(f_\alpha) \leq \textbf{D}_{b - 1, k}$, where $k = \deg(f_{\alpha}).$ Taking the maximum over all values of $k \in \{1, \dots, d\}$ yields a uniform bound that holds for all restrictions $f_\alpha$.
\end{proof}

\begin{coro}\label{bd_cor}For every $f$, and every $d \geq 1$,
\bea\emph{\textbf{D}}(f) &\leq&   \left(\emph{\textbf{D}}_{B_d, d} + \frac{(d+1)B_{d+1}}{2^{d+1}} + \sum_{k = d+2}^{\infty} \frac{k(B_k - B_{k-1})}{2^k}\right) \\ &\leq& \left(\emph{\textbf{D}}_{B_d, d} + \frac{(d+1)^3}{2^{d+1}} + \sum_{k = d+2}^{\infty} \frac{2k^2-k}{2^k}\right).
\eea
\end{coro}

Lemma \ref{bd_recursive} yields explicit bounds on $\textbf{D}_{b, d}$ for any finite $(b,d)$, which in turn yields an explicit bound on $\textbf{D}(f)$ for any $f$ via Corollary \ref{bd_cor}. Incorporating the influence bound $\textbf{D}_{b, d} \leq \frac{d}{2}$, we build up a table of upper bounds $D(b,d)$ recursively, using the rule
\be\label{dynamic}
D(b,d) = \min\left\{\frac{d}{2}, \, \max_{k \in \{1, \dots, d \}}\left\{ d\cdot 2^{-d} + D(b - 1, k) \right\} \right\}.
\ee
Supposing $B_d = d^2$ and extracting bounds recursively already shows that $\textbf{D}(f) < 5.0782$, but we can further improve this by obtaining sharper upper bounds on $B_d$. For values of $d \leq 14$ (which contribute the most to $D(b,d)$ anyway), we can obtain such bounds by manually checking feasibility of a certain linear program, as shown below. (This reduction is inspired by ideas of Nisan and Szegedy in \cite{NS}.)

\begin{fact}\label{bs_reduction}
If there exists a function $f : \{0,1\}^n \to \{0,1\}$ of degree $d$ with block sensitivity $b$, then there exists another function $g: \{0,1\}^b \to \{0,1\}$ of degree $\leq d$ with $g(0) = 0$ and $g(w) = 1$ for each vector $w$ of Hamming weight 1.  
\end{fact}

\begin{proof}
If $f(x)$ attains maximal block sensitivity at $z$, then $f(x \oplus z)$ attains maximal block sensitivity at 0, so without loss of generality we may assume $z = 0$, and possibly replacing $f$ by $1-f$ we may also assume that $f(0) = 0$. If $B_1, \dots, B_b$ are sensitive blocks for $f$ at 0, then define $$g(y_1, \dots, y_b) = f(\underbrace{y_1, \dots, y_1}_{B_1}, \dots, \underbrace{y_b, \dots, y_b}_{B_b}, \underbrace{0, \dots, 0}_{\text{remaining coords}})$$
so that for each coordinate vector $e_i$, $g(e_i) = f(\textbf{1}_{B_i}) = f(0^{B_i}) = 1$.
\end{proof}

For any $d \geq 1$, define the moment map $m_d: \R \to \R^d$ by $m(t) = (t, t^2,\dots, t^d)$.

\begin{prop}\label{bs_LP}
If there exists a degree $d$ function $f: \{0,1\}^n \to \{0,1\}$ with block sensitivity $b$, then there exists $\tau \in \{0,1\}$ such that the following set of linear inequalities has a solution $p \in \R^d$:
\bea \nn
\langle p, m_d(1) \rangle &=& 1 \\ \label{LP}
0 \leq \langle p, m_d(k) \rangle  &\leq & 1 \,\, \text{ for each } k \in \{2, \dots, b-1\} \\ \nn
\langle p, m_d(b) \rangle &=& \tau
\eea
\end{prop}

\begin{proof}
If such an $f$ exists, then let $q(x_1, \dots, x_b) = \frac{1}{b!}\sum_{\sigma \in S_b}g(x_{\sigma(1)}, \dots, x_{\sigma(b)})$, where $g$ comes from Fact \ref{bs_reduction}, and set $\tau = g(1, 1, \dots, 1)$. It is well known (see \cite{BdW}) that there is a univariate polynomial $p: \R \to \R$ of degree at most $d$ such that for any $x \in \{0,1\}^b$, $q(x_1, \dots, x_b) = p(x_1 + \dots + x_b)$. For each $k \in \{1, \dots, b\}$, $p(k)$ is therefore the average value of $g$ on Boolean vectors with Hamming weight $k$, so in particular $p(k) \in [0,1]$. We also know $p(0) = g(0) = 0$, $p(b) = g(1,\dots, 1) = \tau$, and $p(1) = \frac{1}{n}\sum_i g(e_i) = 1$, and hence the coefficients of $p$ provide a solution to the set of linear inequalities.  \end{proof}

Using the simplex method with exact (rational) arithmetic in Maple, we compute the largest $b$ for which the LP (\ref{LP}) is feasible for $1 \leq d \leq 14$, which yields upper bounds on $B_d$ for small $d$. These bounds are summarized in Table \ref{bs table}. Recomputing the table $D(b,d)$ with $B_d$ given by Table \ref{bs table} for $d \leq 14$ (and $B_d = d^2$ for $d > 14$), we can recompute the table as in (\ref{dynamic}) with these boundary conditions. This time $D(30^2, 30) \leq 4.4157\dots$, which implies 
\be
\textbf{D}(f) \leq 4.4158
\ee
for all $f$. If we incorporate the main result of Section \ref{sec_bs_d}, which implies that
$$B_d^2 - B_d \leq \frac{2}{3}(d^4 - d^2)$$
into the table $D(b,d)$, we obtain the slightly stronger result $\textbf{D}_{\infty} \leq 4.3935$, which implies
\begin{theorem}\label{deg_thm}
For all $f$, $n(f) \leq 4.3935 \cdot 2^{\deg(f)}$.
\end{theorem} 

\noindent \textbf{Remark:} The best possible bound that could come from this argument is $\textbf{D}_{\infty} < 3.96$ (if the best known lower bound $B_d \geq d^{\log_3(6)}$ is actually tight, see Section \ref{sec_bs_d}), still quite far from best known lower bound $\textbf{D}_{\infty} \geq 1.5$. Hence to improve significantly upon Theorem \ref{deg_thm} using similar arguments will require improving Lemma \ref{bd_recursive}, which is likely very loose.

\begin{table}\centering
\resizebox{0.8\textwidth}{!}{%
\begin{tabular}{|l|l|l|l|l|l|l|l|l|l|l|l|l|l|l|}
\hline 
$d$ &1 & 2 & 3 & 4 & 5 & 6 & 7 & 8 & 9 & 10 & 11 & 12 & 13 & 14 \\ \hline
$B_d \leq$ &1 & 3 & 6 & 10 & 15 & 21 & 29 & 38 & 47 & 58 & 71 & 84 & 99 & 114 \\ \hline

\end{tabular}%
}

\caption{LP bounds on block sensitivity for low degree functions.}
\label{bs table}
\end{table}

\subsection{Certificate complexity}

Now let us define the analogous quantities for certificate complexity. Recall that $\cert_i(f) = \max_{\{x \, : f(x) \neq f(x^i)\}} C_x(f) + C_{x^i}(f)$. Let $\textbf{C}(f) := \sum_{i \in [n]}\frac{\delta_i(f)}{2^{\cert_i(f)}}$, and for any $H \subseteq [n]$, let $\textbf{C}(H, f) = \sum_{i \in H}\frac{\delta_i(f)}{2^{\cert_i(f)}}$. For any $d \in \N$, we also define $\textbf{C}_d = \max_{\{f \, : \, \deg(f) \leq d\}}\textbf{C}(f)$. 

\begin{theorem}\label{C<1/2}
For any $d \geq 1$, $\emph{\textbf{C}}_d \leq \frac{1}{2}$.
\end{theorem}
\begin{proof}

Let $f$ be a Boolean function with $\deg(f) = d$. For any certificate $C$ for $f$, let $H$ be the set of variables fixed by $C$. It follows from  (\ref{H_rest}) that 
\be \label{C_recur}\textbf{C}(f) \leq \textbf{C}(H, f) + \E_{\alpha \sim \{0,1\}^H}[\textbf{C}(f_{\alpha})].\ee
Since $C$ is a certificate, we know $\deg(f_{\alpha}) \leq d-1$ for all $\alpha$, and $\deg(f_{\alpha^*}) = 0$ for \emph{some} $\alpha^* \in \{0,1\}^H$. So, $\textbf{C}(f_{\alpha^*}) = 0$, and we can improve (\ref{C_recur}) to
\be\textbf{C}(f) \leq \textbf{C}(H, f) + \left(1 - 2^{-|C|}\right)\textbf{C}_{d-1}.\ee
Now take $C$ to be the \emph{globally smallest certificate} for $f$, so that $C_i(f) \geq 2|H|$ for all $i \in \Rel(f)$, and in particular
\be \label{C_recur_improved}
\textbf{C}(f) \leq |H|\cdot 4^{-|H|} + \left(1 - 2^{-|H|}\right) \textbf{C}_{d-1}.
\ee
Since $c\cdot 2^{-c} \leq \frac{1}{2}$ for $c \geq 1$, inequality (\ref{C_recur_improved}) implies that $\textbf{C}_d \leq \alpha \cdot \frac{1}{2} + (1 - \alpha)\cdot \textbf{C}_{d-1}$ for some $\alpha \in [0,1]$. Therefore if $\textbf{C}_{d-1} \leq \frac{1}{2}$ for some $d$, then also $\textbf{C}_d \leq \frac{1}{2}$. The theorem then follows by induction on $d$, noting that $\textbf{C}_1 = \frac{1}{4} < \frac{1}{2}.$
\end{proof}

Since $\textbf{C}(\chi_i(x)) = \frac{1}{4}$, Theorem \ref{C<1/2} cannot be improved by more than a factor of 2. In any case, we have the following immediate corollary (which is tight up to a factor of $\Theta\left(\frac{1}{\sqrt{C(f)}}\right)$ for the monotone address function): 

\begin{theorem}\label{C_thm}
For any $f$, $n(f) \leq \frac{1}{2} \cdot 4^{C(f)}$.
\end{theorem}

Finally, we use an implementation similar to the one above to give a proof of a stronger bound on $n(f)$ in terms of $\deg(f)$ for \textit{monotone} functions $f$.

\begin{theorem}\label{mon_deg}
For monotone functions $f$, $n(f) \leq 1.325 \cdot 2^{\deg(f)}$.
\end{theorem}
\begin{proof}
We let $\widetilde{\textbf{D}}_d$ denote the maximum value of $\textbf{D}(f)$ over all monotone functions of degree at most $d$. Given a monotone $f$ of degree $d$, let $H$ be the variables fixed by any minimal 0-certificate $C$. By monotonicity, $f(0_H, 1_{\overline{H}}) \equiv 0$, so by minimality of $H$, each $i \in H$ must be sensitive for $f$ at the input $(0_H, 1_{\overline{H}})$. Therefore restricting the variables in $\overline{H}$ to 1 yields an $\OR$ function on $H$, and hence each $i \in H$ has $\deg_i(f) \geq |H|$. If we restrict all of the variables in $H$ so that one of the restrictions is constant, we get the analogue of (\ref{C_recur_improved}):
\be
{\textbf{D}}(f) \leq |H| \cdot 2^{-|H|} + \left(1-2^{-|H|}\right)\widetilde{\textbf{D}}_{d-1}.
\ee
However, if we only restrict those variables $i$ in $H$ with $\deg_i(f) = d$, we obtain
\be
{\textbf{D}}(f) \leq |H|\cdot 2^{-d} + \widetilde{\textbf{D}}_{d-1}.
\ee
Combining these two inequalities yields
\be\label{combined_ineq}
\widetilde{\textbf{D}}_d \leq \max_{1\leq k \leq d}\left\{ \min\left(k \cdot 2^{-k} + (1-2^{-k})\widetilde{\textbf{D}}_{d-1}, \, k \cdot 2^{-d} + \widetilde{\textbf{D}}_{d-1}\right)\right\}. 
\ee
Note that $\widetilde{\textbf{D}}_1 = \widetilde{\textbf{D}}_2 = \frac{1}{2}$, since the only monotone functions of degree exactly two are $\AND_2$ and $\OR_2$. Starting with these values and using (\ref{combined_ineq}) to recursively compute bounds on $\widetilde{\textbf{D}}_{d}$, we find that $\widetilde{\textbf{D}}_{30} \leq 1.3243$, and hence ${\textbf{D}}(f) \leq 1.3243 + \sum_{d=31}^{\infty}\frac{d}{2^d} < 1.325$. 
\end{proof}

\noindent \textbf{Remark:} In \cite{CHS}, a function of degree $d$ with $1.5 \cdot 2^d - 2$ relevant variables is constructed. Therefore, Theorem \ref{mon_deg} implies that all monotone functions of a given degree have at least $11\%$ fewer variables than do certain general functions of the same degree. 

\subsection{Sensitivity}

Define $\textbf{S}(f) := \sum_{i \in [n]} \frac{\delta_i(f)}{2^{\sens_i(f)}}$ and $\textbf{S}(H, f) = \sum_{i \in H} \frac{\delta_i(f)}{2^{\sens_i(f)}}$ for any $H \subseteq [n]$. In light of the previous subsections, it seems natural to expect that one should be able to prove a bound $\textbf{S}(f) = O(1)$ for any $f$ using a similar inductive argument, thereby improving Simon's theorem (in the same sense that \cite{CHS} improved Nisan-Szegedy's bound.) However, choosing the right $H$ to restrict for $\textbf{S}$ is tricky business -- neither choice from the previous two subsections will work in general here. (Unfortunately, unlike their counterparts $s(f)$ and $\deg(f)$, the coordinate measures $\sens_i(f)$ and $\deg_i(f)$ are not polynomially related.\footnote{The function $f(x)=x_1 \lor (x_2 \land \cdots \land x_n)$ has $\deg_1(f)=n$ but $\sens_1(f) = 2$.}) Despite this obstruction, we believe such a bound does hold, so we leave it as a conjecture.

\begin{conj}\label{s_conj}
For any $f$, $n(f) \lesssim 4^{s(f)}$. More strongly, $\emph{\textbf{S}}(f) \lesssim 1$. 
\end{conj}

We also remark that our methods from the previous two sections can in fact be used to prove the following ``geometric mean'' of Conjecture \ref{s_conj} with Theorems \ref{deg_thm} and \ref{C_thm}, which directly implies that there does not exist a function $f$ for which both Simon's theorem the bound in either Theorem \ref{deg_thm} or Theorem \ref{C_thm} are asymptotically tight. 
\begin{theorem}\label{mixed}
For any boolean function $f$, 
\bea
n(f) &=& O\left(\sqrt{2^{\deg(f)} \cdot 4^{s(f)}}\right) \\
n(f) &=& \widetilde{O}\left(\sqrt{4^{C(f)} \cdot 4^{s(f)}}\right).
\eea
\end{theorem}

The proof of Theorem \ref{mixed} is based on the observation that the conditions in Definition \ref{RRCM} are convex, and hence the general methods of the preceding sections can be applied to convex combinations of $\sens_i$, $\deg_i$, and $\cert_i$. There are some minor technicalities involved, but we omit the proof for brevity.

\subsection{Decision tree depth}\label{sec_DT}

For decision tree depth -- unlike the other complexity measures considered thus far -- getting a tight bound on $n(f)$ is trivial. Indeed, a depth $d$ binary tree has at most $2^d - 1$ nodes, so $n(f) \leq 2^{\DT(f)} - 1$, and this is obtained by the function which queries a different variable at each node of such a tree. However, the question becomes nontrivial when asked for \emph{monotone} Boolean functions. Let us denote the set of monotone Boolean functions of depth $d$ by $\mathcal{M}_d$ and define the quantities
\be\nn
\textbf{R}^{\DT}_d := \max_{f \in \mathcal{M}_d} n(f) \hspace{5pt} \text{ and } \hspace{5pt} \textbf{R}^{\DT} := \limsup_{d \to \infty} \frac{\textbf{R}^{\DT}_d}{2^d}.
\ee
The best known construction yields a function $f \in \mathcal{M}_d$ with $n(f) = \Theta(d^{-1/2}2^d)$. It is therefore possible that $\textbf{R}^{\DT} = 0$, and we conjecture this to be true (see Section \ref{sec_future_directions}). We provide some evidence in favor of this conjecture by showing that \be\label{mon_DT}\textbf{R}^{\DT} \leq \frac{1}{4}.\ee 

Our proof is based on the observation that unless both of the subfunctions $f_0, f_1$ of a node in a monotone decision tree have very short certificates, they must share a significant number of relevant variables.\footnote{This property, of course, does not hold in general for non-monotone decision trees!} This does not conform to the restriction-induction paradigm of all other arguments in this section, so we carry out the details in the Appendix.

\section[]{A constant factor improvement in the sensitivity conjecture}\label{sec_bs_d}

In their seminal 1992 paper, Nisan and Szegedy \cite{NS} proved an upper bound on the block sensitivity of any Boolean function $f$ in terms of its degree, namely
\be \label{bs_2_deg}
\bls(f) \leq 2 \deg(f)^2.
\ee
In \cite{Tal}, Avishay Tal gives a tensorization argument showing that the constant factor 2 in (\ref{bs_2_deg}) can be reduced to 1:
\be \label{bs_1_deg}
\bls(f) \leq \deg(f)^2.
\ee
In this section, we improve upon the original argument of Nisan and Szegedy to further improve the constant in (\ref{bs_1_deg}):
\begin{theorem}\label{improved_bs_deg}
For any Boolean function $f$, 
\be \label{b^2} \emph{\bls}(f)^2 - \emph{\bls}(f) \leq \frac{2}{3}(\deg(f)^4 - \deg(f)^2) \ee
and hence
\be \label{b^1} \emph{\bls}(f) \leq \sqrt{2/3} \cdot \deg(f)^2 + 1. \ee
\end{theorem}

For many pairs of complexity measures, the proof of the best-known relationships between them make use of the inequality (\ref{bs_1_deg}) as an intermediate step. Upgrading those proofs (see \cite{Midj}, \cite{Huang} and \cite{Nisan}) with Theorem \ref{improved_bs_deg} immediately improves those relations by a constant factor:

\begin{coro}
For any Boolean function $f$,
\bea
\label{improved_sens}\emph{\bls}(f) &\leq& \sqrt{2/3} \cdot s(f)^4 + 1 \\
\emph{\DT}(f) &\leq& \sqrt{2/3} \cdot \deg(f)^3 + \deg(f)\\
C(f) &\leq& \sqrt{2/3} \cdot s(f)^5 + s(f)
\eea
\end{coro}

In particular, (\ref{improved_sens}) improves on Huang's recent result $\bls(f) \leq s(f)^4$, which constitutes the best-known progress on the (strong) sensitivity conjecture, namely $\bls(f) \lesssim s(f)^2$. We note that while the bound in Theorem \ref{improved_bs_deg} can probably be improved further, there is a limit to this approach. The family obtained by tensorizing the function
\bea \nn
f(x_1, \dots, x_6) := \left(\sum_{i=1}^6 x_i\right) - \left(\sum_{1\leq i < j \leq 6} x_ix_j\right) + 
x_1x_3x_4 + x_1x_2x_5 + x_1x_4x_5  \\ \nn + x_2x_3x_4 + x_2x_3x_5 + x_1x_2x_6 + x_1x_3x_6 + x_2x_4x_6 +
x_3x_5x_6 + x_4x_5x_6
\eea
certifies that $\bls(f) \geq \deg(f)^{\log_3(6) \approx 1.63}$ is possible,\footnote{This example is due to Kushilevitz \cite{NW}, and achieves the best-known separation between $\bls(f)$ and $\deg(f)$.} and since Huang's theorem ($\deg(f) \leq s(f)^2$) is tight, combining the two inequalities can never yield a bound stronger than $\bls(f) \leq s(f)^{3.26}$. We also remark that, as a consequence of Theorem \ref{improved_bs_deg}, any function family generated by tensorizing a single example will always have a truly subquadratic separation between $\bls(f)$ and $\deg(f)$. So if it \textit{is} possible to quadratically separate $\bls(f)$ from $\deg(f)$, this will require a different proof technique.





\subsection{Proof of Theorem \ref{improved_bs_deg}}

 We begin by recalling Fact \ref{bs_reduction}, which says that the maximal block sensitivity among functions of degree $d$ is actually obtained by a function $f$ with (i) $f(0) = 0$ and (ii) $f(x) = 1$ for all vectors $x$ of Hamming weight 1. Let us say any $f$ satisfying properties (i) and (ii) is in \emph{standard form}. It is easy to see that any function $f(x)$ in standard form has a real multilinear polynomial expansion which looks like
\be \label{polynomial_f}
f(x_1, \dots, x_b) = x_1 + \cdots + x_b + \sum_{i < j} c_{ij}x_i x_j + \text{ (higher degree terms) }
\ee
where $b = \bls(f) = s(f)$. As it turns out, the coefficients $c_{ij}$ on the quadratic terms $x_ix_j$ in such functions can only take one of two values:
\begin{lem}\label{quad}
If $f(x_1, \dots, x_b)$ is in standard form, then each quadratic term $x_ix_j$ appears with coefficient $c_{ij} \in \{-1, -2\}$ in the polynomial expansion of $f$.
\end{lem}
\begin{proof}
For any pair $i, j$ of coordinates, let $e_{i,j}$ be the vector which has ones in the $i$th and the $j$th coordinates and zeroes elsewhere. Since $f$ is Boolean-valued, $f(e_{i,j}) \in \{0,1\}$. On the other hand, we can compute $f(e_{i,j})$ by plugging into the polynomial (\ref{polynomial_f}), which yields $1 + 1 + c_{ij} \in \{0,1\}$, since all higher degree terms evaluate to 0. 
\end{proof}

If we plug any real numbers $(\mu_1, \dots, \mu_b)$ in $[0,1]^b$ into equation (\ref{polynomial_f}) for the $x_i$, we can interpret the result as the expected value of $f(x)$ where the bits $x_i$ of $x$ are independently sampled Bernoulli$(\mu_i)$'s. In particular, taking all $\mu_i = \mu$, we obtain a univariate polynomial $p_f(\mu)$ whose relevant properties are summarized in the lemma below.
\begin{lem}\label{p_f}
If $f(x_1, \dots, x_b)$ is in standard form, then the polynomial $p_f(\mu)$ satisfies
\begin{enumerate}
    \item $\deg(p_f) \leq \deg(f)$
    \item $\sup_{x \in [0,1]}|p_f(x)| \leq 1$
    \item $|p_f''(0)| \geq b(b-1)$.
\end{enumerate}
\end{lem}
\begin{proof}
Item (1) follows directly from the definition of $p_f$, while item (2) follows from the interpretation of $p_f(\mu)$ as the expected value of the Boolean function $f$. To see (3), observe that (\ref{polynomial_f}) implies that 
\be\label{polynomial_pf}
p_f(\mu) = b\cdot \mu + \left(\sum_{i < j} c_{ij}\right)\mu^2 +\text{ (higher degree terms)},
\ee
and hence by Lemma \ref{quad}, 
\be\nn p''_f(0) = 2\cdot \sum_{i < j} c_{ij} \in \left[-4\binom{b}{2}, -2\binom{b}{2}\right],\ee
which clearly implies (3).
\end{proof}

In light of Lemma \ref{p_f}, to bound $b$ in terms of $\deg(f)$, it suffices to bound $|p_f''(0)|$ in terms of $\deg(p_f)$. This is accomplished by the following fact, which is a direct consequence of V. A. Markov's inequality \cite{Markov}.

\begin{fact}\label{markov_fact}
If $p(x)$ is a degree $d$ polynomial satisfying $0 \leq p(x) \leq 1$ for all $x \in [0,1]$, then
$$|p''(0)| \leq  \frac{2 d^2(d^2-1)}{3}.$$
\end{fact}

\begin{proof}
Recall the famous Markov brothers' inequality, which states that if $q(x)$ is a degree $d$ real polynomial, then for each $k \geq 1$,
$$\sup_{x \in [-1,1]}|q^{(k)}(x)| \leq \frac{d^2(d^2-1^2)(d^2-2^2)\cdots (d^2 - (k-1)^2)}{1 \cdot 3 \cdot 5 \cdots (2k-1)} \sup_{x \in [-1,1]}|q(x)|.$$
In particular, for $k = 2$
\be\label{markov_2}
\sup_{x \in [-1,1]}|q''(x)| \leq \frac{d^2(d^2 - 1)}{3}\sup_{x \in [-1,1]}|q(x)|.
\ee
To translate (\ref{markov_2}) from $[-1,1]$ to $[0,1]$, we simply let $q(x) := \frac{1}{2} - p\left(\frac{1+x}{2}\right).$ Since $x \mapsto \frac{1+x}{2}$ maps $[-1,1]$ to $[0,1]$, we know that $$\sup_{x \in [-1,1]} |q(x)| = \sup_{x \in [0,1]} \left| \frac{1}{2} - p(x)\right| \leq \frac{1}{2}.$$
Similarly, since $q''(x) = -\frac{1}{4} p''(\frac{1+x}{2})$, we also have
\bea
\nn|p''(0)| \leq \sup_{x \in [0,1]} |p''(x)| &=& 4\sup_{x \in [-1,1]}|q''(x)|\\ \nn &\leq& 4\cdot \frac{d^2(d^2-1)}{3} \cdot \sup_{x \in [-1,1]}|q(x)| \\ \nn &\leq& \frac{2 d^2(d^2-1)}{3}.
\eea
\end{proof}
Combining (1), (2), and (3) from Lemma \ref{p_f} with Fact \ref{markov_fact} yields
(\ref{b^2}). This then implies (\ref{b^1}), because if $b^2 - b \geq \frac{2(d^4 - d^2)}{3}$, then
$$(b-1)^2 \leq b^2 - b \leq (2/3)\cdot (d^4 - d^2) \leq (2/3)\cdot d^4,$$
and taking square roots gives Theorem \ref{improved_bs_deg}.  \qed

\subsection{Block sensitivity vs. approximate degree:} 
In \cite{NS}, the authors also prove a bound on block sensitivity in terms of the \emph{approximate} degree, namely
\be\label{bs_2_adeg} \bls(f) \leq 6 \cdot (\widetilde{\deg}_{1/3}(f))^2. \ee
Again we can streamline their argument to improve the constant, this time from $6$ to $5$. We remark that, although $\adeg(f \circ g) = O(\adeg(f) \cdot \adeg(g))$ (by a result of Sherstov \cite{Sherstov}), the implicit constant in the $O(\cdot)$ obstructs us from reducing the constant in (\ref{bs_2_adeg}) to 1 via tensorization. Another difference between (\ref{bs_2_adeg}) and (\ref{bs_2_deg}) is that (\ref{bs_2_adeg}) is known to be tight up to the constant -- it is shown in \cite{NS} that $\OR_n$ can be $1/3$-approximated by a Chebyshev polynomial of degree $2\sqrt{n}$, and hence the 6 cannot be replaced by anything smaller than $\frac{1}{4}$ in (\ref{bs_2_adeg}).

\begin{theorem}\label{approx}
For any Boolean function $f$, 
$$\emph{\bls}(f) \leq 5 \cdot (\widetilde{\deg}_{1/3}(f))^2$$
\end{theorem}

\begin{proof} By reasoning as in Fact \ref{bs_reduction}, we may assume that $f$ is in standard form with (block) sensitivity $b$. Let $p(x_1, \dots, x_b)$ be a polynomial of degree $d = \widetilde{\deg}_{1/3}(f)$ satisfying $|p(x) - f(x)| \leq 1/3$ for all $x \in \{0,1\}^b$. Write
\be \nn
p(x) = c_0 + c_1x_1 + \cdots + c_bx_b + (\text{higher order terms}),\ee
and observe that 
\bea
|p(0) - f(0)| &\leq& 1/3 \implies |c_0| \leq 1/3\\
|p(e_i) - f(e_i)| &\leq& 1/3 \implies c_i + c_0 \geq 2/3.\eea 
Therefore each $c_i \geq 1/3$, and so $\sum_{i = 1}^b c_i \geq b/3$. Viewing $p$ as a function on $[0,1]^b$ via its multilinear extension, and considering the univariate function $q(t) := \frac{1}{2} - p(\frac{1+t}{2}, \frac{1+t}{2}, \dots, \frac{1+t}{2})$, we have that 
$$\sup_{-1 \leq t \leq 1} |q(t)| = \sup_{0 \leq t \leq 1} \left|\frac{1}{2} - p(t, t, \dots, t)\right| = \sup_{x \in [0,1]^b} \left|\frac{1}{2} - p(x)\right| \leq \frac{1}{2} + \frac{1}{3} = 5/6, $$
where the middle inequality is due to convexity/multilinearity. On the other hand, $q'(0) = \frac{1}{2} \sum_{i=1}^b (\partial_i p)(0) = \frac{1}{2}\sum_{i=1}^s c_i \geq b/6.$ By Markov's inequality (in the $k=1$ case), this implies 
\be \nn b/6 \leq \frac{5d^2}{6} \implies b \leq 5d^2.\ee \end{proof}

\section{Open problems and future directions}\label{sec_future_directions}

In addition to Conjecture \ref{s_conj}, we suggest some other questions left open by our work: \\

\noindent \textbf{Asymptotically stronger bounds on $n(f)$ for monotone functions:} For monotone functions $f$, our work shows stronger bounds on $n(f)$ in terms of $\deg(f)$, $s(f)$ and $\DT(f)$ than are known for general functions. However, these bounds still fall short of the best known construction. The best known construction for each of the three measures above is due to Wegener \cite{Weg}, which we now briefly describe. For each odd integer $k \geq 1$, we define the \emph{monotone address function}
$$\text{MAF}_{k}\left(x_1, \dots, x_k, \left\{y_S\right\}_{S \in \binom{[k]}{\lfloor k/2 \rfloor}}\right) := \MAJ(x_1, \dots, x_k) \bigvee_{ S \in \binom{[k]}{\lfloor k/2 \rfloor}} \left( \bigwedge_{i \in S} x_i \land y_S \right).$$
It isn't hard to show that for $f = \MAF_k$, 
$$n(f) = \Theta\left(\frac{2^{\DT(f)}}{\sqrt{\DT(f)}}  \right) = \Theta\left(\frac{2^{\deg(f)}}{\sqrt{\deg(f)}}\right) = \Theta\left(\frac{4^{s(f)}}{\sqrt{s(f)}}\right) = \Theta\left(\frac{4^{C(f)}}{\sqrt{C(f)}}\right).$$
We conjecture that this is the best possible for monotone functions. \\

\noindent \textbf{Approximate junta size:} If $s(f) = s$, then is $f$ $\varepsilon$-close to a $O_{\varepsilon}(4^s)$ junta? Verbin, Servedio and Tan conjectured that for \emph{monotone} $f$ with $\DT(f) = d$, $f$ must be $\varepsilon$-close to a $\poly_{\varepsilon}(d)$ junta, which would imply the same for $s(f)$. However, Kane \cite{Kane} showed this was false, by constructing a (random) monotone function with $\DT(f) = d$ which is not $0.1$-close to any $\exp(\sqrt{d})$-junta. This is tight up to a constant in the exponent by Freidgut's theorem and the OS inequality ($\I[f] \leq \sqrt{\DT(f)}$ for monotone $f$, see \cite{OS}). Since $s(f) \leq \DT(f)$, Kane's construction is also a monotone function with $s(f) = s$ that is not $0.1$-close to any $\exp(\sqrt{s})$-junta. \\

\noindent \textbf{Do large juntas have smaller separations?} If $n(f)$ is exponential in $s(f)$, $\deg(f)$, $C(f)$, and $\DT(f)$, then how are these measures related? For example, if $n(f) = 2^{\Omega(\deg(f))}$ then $s(f) = \Omega(\deg(f))$, by Simon's theorem; if $n(f) = 2^{\Omega(s)}$, then $\deg(f) = \Omega(s(f))$ by Nisan-Szegedy. Do the other directions hold? What can be said if $n(f) \geq 2^{s(f)^{1/100}}$? Is it possible for such ``large" functions to exhibit $\omega(1)$-separations between complexity measures? \\

\appendix 
\section*{Appendix}

Here we carry out the argument outlined in Section \ref{sec_DT} to prove that a monotone decision tree can depend on at most $2^{d-2} + 2$ variables.

\begin{lem}\label{DT_intersect} Let $f_0$, $f_1$ be the two subfunctions from the root node in a monotone decision tree. If neither $f_0$ nor $f_1$ is constant, then
$$C_{\min}^0(f_0) + C_{\min}^1(f_1) \leq |\emph{\Rel}(f_0) \cap \emph{\Rel}(f_1)| + 1.$$

\end{lem}
\begin{proof}
We first claim that every assignment to $\Rel(f_0) \cap \Rel(f_1)$ must either force $f_0 = 0$ or force $f_1 = 1$.
To see this, let $C:= \Rel(f_0) \cap \Rel(f_1)$. Let us decompose any assignment $\alpha$ to $\Rel(f)$ into $(\alpha_x, \alpha_0, \alpha_1, \alpha_C)$, where each component is the assignments to $x$ (the root node), $\Rel(f_0) \setminus \Rel(f_1)$, $\Rel(f_1) \setminus \Rel(f_0)$, and $C$ respectively. Suppose for the sake of contradiction that there is some assignment $\beta_C$ to $C$ which does not force $f_0 = 0$ or $f_1 = 1$ -- then we can pick assignments $\alpha, \alpha'$ such that: (i) $f_0(\alpha_0, \beta_C) = 1$ and (ii) $f_1(\alpha'_1, \beta_C) = 0$. But then $f(0, \alpha_0, \alpha'_1, \beta_C) = 1$ and $f(1, \alpha_0, \alpha'_1, \beta_C) = 0$, which violates monotonicity of $f$ since $(0, \alpha_0, \alpha'_1, \beta_C) \prec (1, \alpha_0, \alpha'_1, \beta_C)$, which proves our claim.

Now fix some ordering $x_1, \dots, x_{|C|}$ of $C$ and consider the $|C|+1$ assignments $$\alpha_i := (\underbrace{1, \dots, 1}_{i}, \underbrace{0 \dots, 0}_{|C| - i}), \text{  for } i = 0, 1, \dots, |C|.$$ By the claim above, each $\alpha_i$ forces either $f_0 = 0$ or $f_1 = 1$. In particular, we know that $\alpha_0$ forces $f_0 = 0$ and $\alpha_{|C|}$ forces $f_1 = 1$. (Indeed, if $\alpha_0$ does not force $f_0 = 0$, then $f_1 \equiv 1$, which we assumed is not the case, and likewise for $\alpha_{|C|}$.) Therefore, since $\alpha_i \prec \alpha_{i+1}$, there must be some $0 \leq i \leq |C| - 1$ for which $\alpha_i$ forces $f_0 = 0$ and $\alpha_{i+1}$ forces $f_1 = 1$. Hence, by monotonicty, there is a 1-certificate for $f_1$ fixing only the variables $\{x_1, \dots, x_{i+1}\}$ to 1, and a 0-certificate for $f_0$ fixing only the variables $\{x_{i+1}, \dots, x_{|C|}\}$ to 0. This implies $C_{\min}^0(f_0) + C_{\min}^1(f_1) \leq i + 1 + |C| - i = |C| + 1$. 
\end{proof}

We also need the following standard fact, whose easy proof we omit:

\begin{fact}\label{DT and/or}
Let $g$ be any function which does not depend on the variable $a$. Then $\emph{\DT}(a \lor g)= \emph{\DT}(a \land g) = 1 + \emph{\DT}(g)$.
\end{fact}

\begin{lem}\label{DT_bound}
For $d \geq 2$, $\emph{\textbf{R}}^{\emph{\DT}}_d \leq \max\left\{2 \emph{\textbf{R}}^{\emph{\DT}}_{d-1} - 2,  \,2+2\emph{\textbf{R}}^{\emph{\DT}}_{d-2}, \, 1+\emph{\textbf{R}}^{\emph{\DT}}_{d-1} \right\}.$
\end{lem}
\begin{proof}
Let $f \in \mathcal{M}_d$, for $d \geq 2$. We consider the possible values of $c_0 := C_{\min}^0(f_0)$ and  $c_1:= C_{\min}^1(f_1)$. If either $c_0 = 0$ or $c_1 = 0$ (i.e. one of the subfunctions is constant), then $n(f) \leq 1 + \textbf{R}_{d-1}^{\DT}$. 

Otherwise, $c_0, c_1 \geq 1$ and Lemma \ref{DT_intersect} applies. If $c_0 + c_1 \geq 4$, then by the lemma, 
\bea \nn n(f) &\leq& 1+ n(f_0) + n(f_1) - |\Rel(f_0) \cap \Rel(f_1)| \\ &\leq& n(f_0) + n(f_1) - 2 \\ \nn &\leq& 2\textbf{R}_{d-1}^{\DT} - 2.
\eea
If $c_0 = c_1 = 1$, then we can write $f_0 = a \land g$ and $f_1 = a \lor h$ for some functions $g$ and $h$ which do not depend on $a$. By Fact \ref{DT and/or}, $\DT(g) = \DT(f_0) - 1 \leq d-2$, and similarly $\DT(h) \leq d-2$, so 
$n(f) \leq 1 + 1 + n(g) + n(h) \leq 2 + 2\textbf{R}^{\DT}_{d-2}.$

Finally, it remains to consider the case when $\{c_0, c_1\} = \{1,2\}$. Without loss of generality, suppose $c_0 = 1$. It follows that we can write $f_0 = a \land g$, for some function $g$ which does not depend on $a$. By Fact \ref{DT and/or}, $\DT(g) = \DT(f_0) - 1 \leq d-2$, and hence 
\bea \nn
n(f) &\leq& 1 + n(f_0) + n(f_1) - |\Rel(f_0) \cap \Rel(f_1)|\\ 
&\leq& 1 + \textbf{R}_{d-1}^{\DT} + 1 + \textbf{R}_{d-2}^{\DT} - 3 \\ \nn
&=& {\textbf{R}}^{{\DT}}_{d-1} + {\textbf{R}}^{{\DT}}_{d-2} - 1 \\ \nn &\leq& 2{\textbf{R}}^{{\DT}}_{d-1} - 2.
\eea
\end{proof}

\noindent\textit{Proof of (\ref{mon_DT}):} Since $\textbf{R}_{1}^{\DT} = 1$, Lemma \ref{DT_bound} immediately implies $\textbf{R}_{2}^{\DT} \leq 2$, $\textbf{R}_{3}^{\DT} \leq 4$, $\textbf{R}_{4}^{\DT} \leq 6$, and $\textbf{R}_{5}^{\DT} \leq 10$. It is also easy to construct explicit examples showing that these all of these inequalities are actually equalities -- in fact, if $g(x) \in \mathcal{M}_{d-2}$, then the function $$f(a,b,x,y) = ((\neg a) \land (b \land g(x))) \lor (a \land (b \lor g(y))) \in \mathcal{M}_d$$ has $2n(g) + 2$ relevant variables. Therefore $\textbf{R}^{\DT}_d \geq 2\textbf{R}^{\DT}_{d-2} + 2$, and the bound $\textbf{R}^{\DT}_d \leq 2\textbf{R}^{\DT}_{d-1} - 2$ becomes the dominant bound in the lemma for $d \geq 4$. We can rewrite this inequality as $$(\textbf{R}^{\DT}_d - 2) \leq 2(\textbf{R}^{\DT}_{d-1} - 2) \, \, \text{ for } d \geq 5,$$ and therefore $(\textbf{R}^{\DT}_d - 2) \leq 2^{d-5}(10-2) = 2^{d-2} \implies \textbf{R}^{\DT}_d \leq 2^{d-2} + 2$ and $\textbf{R}^{\DT} \leq \frac{1}{4}$. \qed


\bibliographystyle{amsplain}


\begin{dajauthors}
\begin{authorinfo}[jake]
  Jake Wellens\\
  narf1899\imageat{}gmail\imagedot{}com \\
\end{authorinfo}
\end{dajauthors}

\end{document}